\documentclass[12pt, a4]{amsart}
\usepackage{amscd, amsmath, amssymb}
\usepackage{fullpage}
\usepackage{mathrsfs,epsfig}
\usepackage{enumitem}

\theoremstyle{plain}
\newtheorem{thm}{Theorem}[section]

\theoremstyle{definition}
\newtheorem{defn}[thm]{Definition}
\newtheorem{ex}[thm]{Example}
\newtheorem{rem}[thm]{Remark}

\numberwithin{equation}{section}

\newcommand{\R}{\mathbb{R}}

\usepackage{color}

\begin{document}

\title[An affine Birkhoff--Kellogg type result]{An affine Birkhoff--Kellogg type result in product spaces and its application to differential systems} 

\date{}

\author[A. Calamai]{Alessandro Calamai}
\address{Alessandro Calamai, 
Dipartimento di Ingegneria Civile, Edile e Architettura,
Universit\`{a} Politecnica delle Marche
Via Brecce Bianche
I-60131 Ancona, Italy}%
\email{a.calamai@univpm.it}%

\author[G. Infante]{Gennaro Infante}
\address{Gennaro Infante, Dipartimento di Matematica e Informatica, Universit\`{a} della
Calabria, 87036 Arcavacata di Rende, Cosenza, Italy}%
\email{gennaro.infante@unical.it}%

\author[G. A. Veltri]{Giuseppe Antonio Veltri}
\address{Giuseppe Antonio Veltri, 
Dipartimento di Scienze Matematiche e Informatiche, Scienze Fisiche e Scienze della Terra, Universit\`{a} di Messina
98166 Messina, Italy}%
\email{giuseppe.veltri@studenti.unime.it}%

\begin{abstract} 
We prove a version of the Birkhoff-Kellogg theorem in product spaces, under affine transformations. This is a fairly natural framework that arises when dealing with parameter-dependent systems of functional-differential equations. We provide criteria on the existence of parameter-dependent solutions that have all the components nontrivial. The theoretical results are illustrated in the case of systems of second order functional-differential equations, where the functional part can cover the interesting cases of time and state-dependent deviated arguments. We provide a concrete example, for which we furnish numerically approximated solutions, coherent with our theoretical framework, and in which we compute or estimate all the constants that are required by our abstract results.
\end{abstract}



\maketitle

\section{Introduction}
The classical Birkhoff-Kellogg theorem~\cite{B-K-1922} and its variants, play an interesting role when dealing with the existence of eigenvalues and eigenfunctions in infinite-dimensional normed linear spaces. A remarkable feature of these results is their direct applicability to the solvability of parameter-dependent problems in the context of differential equations, see for example the classical books~\cite{guolak, Krasno} and the recent papers~\cite{acgi2, gi-BK}.
In particular, in the manuscript~\cite{acgi2}, the authors provided a version of the Birkhoff-Kellogg theorem in \emph{affine cones}. One motivation for the affine setting relies on the fact that it arises quite naturally when dealing with differential equations in the presence of delays. On the other hand, the solvability of vectorial equations has seen renovated interest, and new Birkhoff-Kellogg type results in product spaces have been recently developed in the paper~\cite{acgijrl}. One advantage of the application  of the vectorial results in~\cite{acgijrl}, with respect to a direct application of a Birkhoff-Kellogg type theorem for operator equations, is that it yields better localization results, see Remark~2.8 of~\cite{acgijrl}. Here, essentially, we propose an affine reformulation of the results in~\cite{acgijrl}. More specifically, we investigate parametric systems of type
	\begin{equation} \label{sys_intro}
		\left\{\begin{array}{l}
        x_1=y_1+\lambda_1\, F_1(x_1,x_2), \\ x_2=y_2+\lambda_2\,F_2(x_1,x_2), \end{array} \right.
	\end{equation}
where $y_1$ and $y_2$ are given, and the compact operator $F=(F_1,F_2)$ is defined on the Cartesian product of two sets $y_1+C_1$ and $y_2+C_2$, which can be as follows:
\begin{enumerate}
	\item $C_1$ and $C_2$ are cones; 
	\item $C_1$ is a cone and $C_2$ is a disk in an infinite dimensional normed space; 
	\item both $C_1$ and $C_2$ are disks in infinite dimensional normed spaces. 
\end{enumerate}
Under natural conditions, we
show the existence of solution pairs $(x_1,x_2)$ of the fixed point system \eqref{sys_intro}, corresponding to real parameters $\lambda_1$, $\lambda_2$. Our approach allows to prescribe the norm of each component $x_1$ and $x_2$ of the solution pair, which is therefore  nontrivial in a strong sense.

We wish to point out that a similar component-wise approach has been pursued by many authors in 
the context of fixed point theory; see for example~\cite{Avramescu, Benedetti, imap, InMaRo, Perov, PrecupFPT}.
On the other hand, the theory we develop here is more related to nonlinear spectral theory (see e.g.~\cite{ADV}), even if, formally, our results do not give the existence of ``eigenvalues'' and ``eigenfunctions'' since here we
work in an affine context,
in the spirit of~\cite{acgi2}: see also the contributions by Djebali and Mebarki~\cite{djeb2014}.
In a sense here we combine the component-wise approach with the affine setting, obtaining a quite general abstract existence result which therefore finds applications for a wide class of functional differential systems. 

Our approach is motivated by the applications.
To illustrate this in a simple setting, we take into consideration the model of a bended rod presented by Figueroa and Pouso~\cite{rub-rod-lms} and 
consider two wires of length 3, bent in an S-shaped way, placed in proximity of each other, and heated.
To fix the ideas, we set the spatial variable $t$ in the interval $[-1,2]$.
The two wires in the intervals $[-1,0] \cup [1,2]$ are kept at a certain temperature, say $\psi_i(t)$ for $t \in [-1,0]$ and $\varphi_i(t)$ for $t \in [1,2]$, for $i=1,2$, see Figure~\ref{rods}.
\begin{figure}[h]
    \centering
    \includegraphics[width=0.62\linewidth]{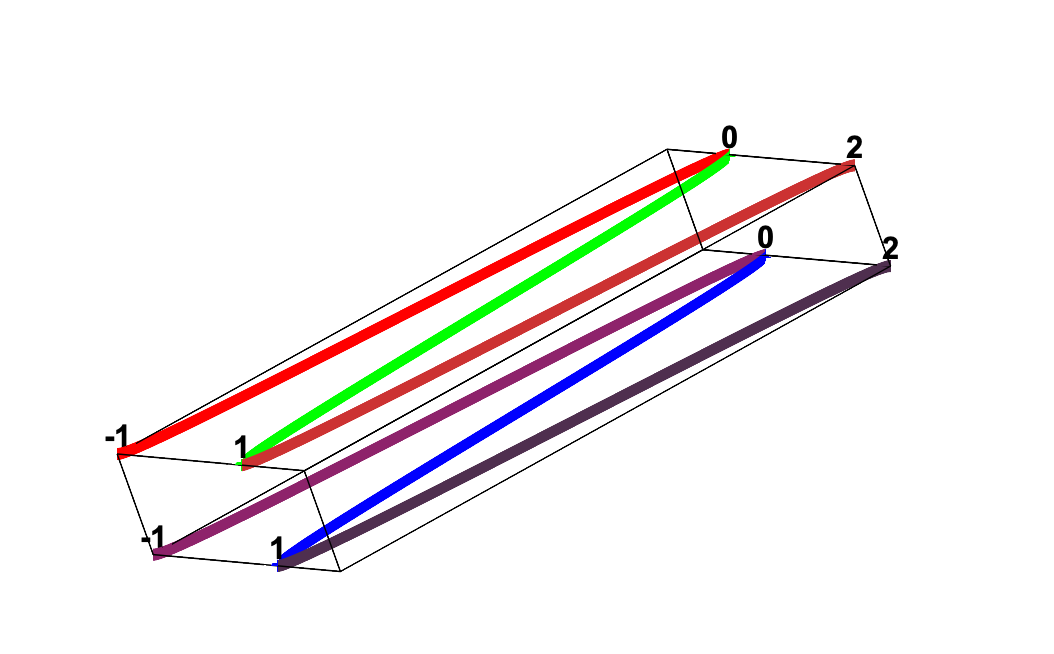}
    \caption{Two heated rods}
    \label{rods}
\end{figure}
Due to the bending and the proximity of the two wires, in the interval $[0,1]$ the temperature at a point $t$ of one of the wires is influenced both by the values of the temperatures at the points $-t$ and $1+t$ and by the value of the temperature of the other wire at the same point $t$.
This leads to the following differential system ($i=1,2$):
\begin{equation} \label{sys_diff_intro}
\left\{\begin{array}{rll}
-u_{i}''(t) &= f_{i}(t, u_{i}(t), h_i(u_i(-t),u_i(t+1),u_j(t))),\, (i\neq j), \quad & t \in (0, 1), \\
u_{i}(t) &= \psi_{i}(t), \quad & t \in [-1, 0], \\
u_{i}(t) &= \varphi_{i}(t), \quad & t \in [1,2],
\end{array} \right.
\end{equation}
where $h_i$ are given continuous functions that reflect the interactions in the system.
Here we study the solvability of a generalization of the system~\eqref{sys_diff_intro}
set in a given interval $[-\sigma_i,\tau_i]$, $(i=1,2)$, namely
\begin{equation} \label{sys_diff_BIS}
\left\{\begin{array}{rll}
-u_{i}''(t) &= \lambda_{i}f_{i}(t, u_{1}(t), u_{2}(t), H_{i}(u_1,u_2)(t)), \quad & t \in (0, 1), \\
u_{i}(t) &= \psi_{i}(t), \quad & t \in [-\sigma_i, 0], \\
u_{i}(t) &= \varphi_{i}(t), \quad & t \in [1, \tau_i],
\end{array} \right.
\end{equation}
where $\lambda_{i} \in \R$, $H_{i}$ is a suitable operator, $f_i, \psi_{i}, \varphi_{i}$ are continuous functions, with $f_i$ possessing a suitable growth. We stress that the operator $H_{i}$ that occurs in~\eqref{sys_diff_BIS} is well tailored in order to take into account several effects, for example time and space-dependent delays, advanced or deviated arguments.

Functional-differential equations are a well-investigated subject, see, e.g.~\cite{halelunel};
for functional differential equations with state-dependent delays, see, e.g., the survey~\cite{hartung} and the recent manuscript~\cite{hemedr}.
In such a theory, equations with deviated arguments play an interesting role: for the second order case we refer the reader to the classical papers by Grimmm and Schmitt \cite{GrimmSchmitt} and Ntouyas and Tsamatos~\cite{NtouyasTsamatos1,NtouyasTsamatos2}.
We prove the existence of solutions for the system~\eqref{sys_diff_BIS} that have component-wise (non-zero) fixed norms and corresponding associated vectorial parameters that can be estimated. We show, in an example, that all the constants that occur in our theory can be either computed or estimated; furthermore we provide numerically approximated solutions and parameters, which are coherent with the theoretical findings.

\section{Birkhoff-Kellogg type results on translates}

In this section we provide a vectorial and affine version of the Birkhoff-Kellogg theorem.
From now on, we will denote by $(X,\| \ \|)$ a normed linear space and by $K$ a cone in $X$;
that is, a closed set with $K+K\subset K$, $\mu K\subset K$ for all $\mu\ge 0$ and $K\cap(-K)=\{0\}$.

Moreover, by $B_{r}$ we will mean the open ball in $X$ centered at the origin and with radius $r>0$, while
by $\overline{B}_{r}$, resp.\ $\partial\overline{B}_{r}$, we denote the closed disk and its boundary, respectively.
We also introduce the following notation: $K_{r}= B_{r}\cap K$, so that $\overline{K}_{r}= \overline{B}_{r} \cap K$, resp.\ 
$\partial\overline{K}_{r}= \partial\overline{B}_{r} \cap K$, denote the closure and boundary of $K_{r}$ relative to $K$.
Notice that $\partial  \overline{K}_{r}$ is a retract of $\overline{K}_{r}$. 
With a slight abuse of notation (note that the whole space $X$ is not a cone) we still denote $X_{r}= B_{r}$, $\overline{X}_{r}= \overline{B}_{r}$,  
$\partial\overline{X}_{r}= \partial\overline{B}_{r}$. Observe that, if $X$ is infinite dimensional, $\partial  \overline{X}_{r}$ is a retract of $\overline{X}_{r}$.

Let $(X_1,\|\ \|_1)$ and $(X_2,\|\ \|_2)$ be normed linear spaces and $C_1\subset X_1$, $C_2\subset X_2$ such that for each $i\in\{1,2\}$ either
\begin{enumerate}
	\item[(a)] $C_i=K_i$ is a cone; or
	\item[(b)] $C_i = X_i$ is an infinite dimensional normed space. 
\end{enumerate}
We will consider the product space $X_1 \times X_2$ endowed with the maximum norm.

The following result (see \cite[Theorem 2.3]{acgijrl}) is a version of the Birkhoff--Kellogg Theorem in product spaces.

\begin{thm}\label{th_sys}
	Let $r_1,r_2$ be positive constants and suppose that $$T=(T_1,T_2):\overline{C}_{1,r_1}\times \overline{C}_{2,r_2}\longrightarrow C_1\times C_2$$ is a compact map satisfying that 
		\begin{equation*}
\inf_{\left\|x_1\right\|_1=r_1,\ \left\|x_2\right\|_2= r_2}\left\|T_1(x_1, x_2)\right\|_1>0 \quad \text{and} \quad \inf_{\left\|x_1\right\|_1= r_1,\ \left\|x_2\right\|_2= r_2}\left\|T_2(x_1, x_2)\right\|_2>0.
	\end{equation*}
	Then there exist $\lambda_1,\lambda_2>0$ and $(x^*_1, x^*_2)\in \overline{C}_{1,r_1}\times \overline{C}_{2,r_2}$ with $\left\|x^*_1\right\|_1=r_1$ and $\left\|x^*_2\right\|_2=r_2$ such that
	\begin{equation*}
		\left\{\begin{array}{l} x^*_1=\lambda_1\, T_1(x^*_1, x^*_2), \\ x^*_2=\lambda_2\,T_2(x^*_1, x^*_2). \end{array} \right.
	\end{equation*}
\end{thm}

Given $y_i \in X_{i}$ and $r_i>0$, $i=1,2$, we introduce the notation
$$
C_{i,r_i}^{y_i} := \{x_i = y_i + w_i : w_i \in C_{i,r_i} \}, i=1,2.
$$
We now present a variation of Theorem~\ref{th_sys} in the context of translates in product spaces.
\begin{thm}\label{th_affine_sys}
	Let $y_1 \in X_{1}, y_2 \in X_{2}$ be given, let $r_1,r_2$ be positive constants and suppose that $$F=(F_1,F_2):\overline{C_{1,r_1}^{y_1}}\times \overline{C_{2,r_2}^{y_2}}\longrightarrow C_1\times C_2$$ is a compact map satisfying that
        \begin{equation*}
    \inf_{x_1 \in \partial C_{1,r_1}^{y_1}, x_2 \in \partial{C_{2,r_2}^{y_2}}} \|F_1(x_1, x_2)\|_1>0 \quad \text{and} \quad
    \inf_{x_1 \in \partial{C_{1,r_1}^{y_1}}, x_2 \in \partial C_{2,r_2}^{y_2}} \|F_2(x_1, x_2)\|_2 > 0.
	\end{equation*}
	Then there exist $\lambda_1,\lambda_2>0$ and $(x^*_1, x^*_2)=(y_1+w^*_1,y_2+w^*_2)\in \overline{C_{1,r_1}^{y_1}}\times \overline{C_{2,r_2}^{y_2}}$ with $\left\|w^*_1\right\|_1=r_1$ and $\left\|w^*_2\right\|_2=r_2$ such that
	\begin{equation}\label{eq_sol_sys}
		\left\{\begin{array}{l} x^*_1= y_1 +\lambda_1\, F_1(x^*_1, x^*_2), \\ x^*_2=y_2 +\lambda_2\,F_2(x^*_1, x^*_2). \end{array} \right.
	\end{equation}
\end{thm}

\begin{proof}
    Consider the following system:
\begin{equation*}
    \begin{cases}
        x_1 = y_1 + \lambda_1 F_1(x_1, x_2), \\
        x_2 = y_2 + \lambda_2 F_2(x_1, x_2).
    \end{cases}
\end{equation*}
By making the substitution $w_i = x_i - y_i$ (and thus $x_i = y_i + w_i$), we obtain:
\begin{equation*}
    \begin{cases}
        w_1 = \lambda_1 F_1(y_1 + w_1, y_2 + w_2), \\
        w_2 = \lambda_2 F_2(y_1 + w_1, y_2 + w_2).
    \end{cases}
\end{equation*}
Defining the operators $G_i(w_1, w_2) := F_i(y_1 + w_1, y_2 + w_2)$, the system takes the form:
\begin{equation*}
    \begin{cases}
        w_1 = \lambda_1 G_1(w_1, w_2), \\
        w_2 = \lambda_2 G_2(w_1, w_2),
    \end{cases}
\end{equation*}
and the result now follows from Theorem~\ref{th_sys}.
\end{proof}

    We stress that when the operator $F$ is defined in the product of an affine cone times an infinite dimensional normed space, an additional solution can be obtained.
    In fact the following result holds, see
    Theorem 2.5 in \cite{acgijrl}.

\begin{thm} \label{th_sys_kb}
	Let $K_1$ be a cone in the normed linear space $X_1$, and $X_2$ be an infinite dimensional normed space. 
	Let $y_1 \in X_{1}, y_2 \in X_{2}$ be given, let $r_1,r_2$ be positive constants and suppose that
$$F=(F_1,F_2):\overline{K_{1,r_1}^{y_1}}\times \overline{B_{2,r_2}^{y_2}}\longrightarrow K_1\times X_2$$ is a compact map satisfying that
        \begin{equation*}
    \inf_{x_1 \in \partial K_{1,r_1}^{y_1}, x_2 \in \partial B_{2,r_2}^{y_2}} \|F_1(x_1, x_2)\|_1>0 \quad \text{and} \quad
    \inf_{x_1 \in \partial K_{1,r_1}^{y_1}, x_2 \in \partial B_{2,r_2}^{y_2}} \|F_2(x_1, x_2)\|_2 > 0.
	\end{equation*}
Then there exist $\lambda_{1,1},\lambda_{2,1},\lambda_{1,2} >0$, $\lambda_{2,2}<0$ and $(x^*_{1,j}, x^*_{2,j})=(y_1+w^*_{1,j},y_2+w^*_{2,j})\in \overline{K_{1,r_1}^{y_1}}\times \overline{B_{2,r_2}^{y_2}}$, $j=1,2$,  with $\left\|w^*_{1,j}\right\|_1=r_1$ and $\left\|w^*_{2,j}\right\|_2=r_2$ such that
	\begin{equation*}
		\left\{\begin{array}{l} x^*_{1,j}= y_1 +\lambda_{1,j}\, F_1(x^*_{1,j}, x^*_{2,j}), \\ x^*_{2,j}=y_2 +\lambda_{2,j}\,F_2(x^*_{1,j}, x^*_{2,j}). \end{array} \right. \qquad (j=1,2).
	\end{equation*}    
\end{thm} 

\begin{rem}
Note that, under the assumptions of Theorem \ref{th_affine_sys}, if both $C_1$ and $C_2$ are
infinite dimensional normed spaces, then there exist four couples of numbers
$\lambda_1,\lambda_2$ and associated points $(x^*_1, x^*_2)$ such that \eqref{eq_sol_sys}
holds; see Remark 2.6 in \cite{acgijrl}.
\end{rem}

\section{An application to systems of functional differential equations}

Motivated by \eqref{sys_diff_intro},
we investigate a differential system with functional terms and subject to initial and final conditions.
In this section, for $i=1,2$, we fix $\sigma_i>0$ and $\tau_i>0$ and denote by
$(X_i,\|\ \|_i)$ the Banach space $C[-\sigma_i, \tau_i]$ endowed with the usual supremum norm.
We consider the system ($i=1,2$)
\begin{equation} \label{sys_diff}
\left\{\begin{array}{rll}
-u_{i}''(t) &= \lambda_{i}f_{i}(t, u_{1}(t), u_{2}(t), H_{i}(u_1,u_2)(t)), \quad & t \in (0, 1), \\
u_{i}(t) &= \psi_{i}(t), \quad & t \in [-\sigma_i, 0], \\
u_{i}(t) &= \varphi_{i}(t), \quad & t \in [1, \tau_i],
\end{array} \right.
\end{equation}
where $\lambda_{i} \in \R$,
the operator $H_{i}$ is defined in a suitable subset of $X_{1} \times X_{2}$, 
the real valued function $f_i$ is defined in a subset of $\R^4$, and the initial and final conditions, respectively, are continuous functions
$\psi_{i}:[-\sigma_i, 0]\to\R$
and
$\varphi_{i}: [1, \tau_i]\to\R$.

In order to apply the results of the previous section, we rewrite system \eqref{sys_diff} in an equivalent integral form. Namely,
\begin{equation} \label{sys_integral}
u_{i}(t) = \phi_{i}(t) + \lambda_{i} \int_{0}^{1} k_{i}(t,s) f_{i}(s, u_{1}(s), u_{2}(s), H_{i}[u_{1},u_{2}](s)) \,ds,\ i=1,2,\ t\in [-\sigma_i,\tau_i].
\end{equation}
Here the function $\phi_i:[-\sigma_i,\tau_i]\to\R$ is defined as
\[
\phi_i(t) = 
\begin{cases} 
\psi_{i}(t), & t \in [-\sigma_i, 0], \\
(1-t)\psi_{i}(0)+t\varphi_{i}(1), & t \in (0, 1), \\
\varphi_{i}(t), & t \in [1, \tau_i],
\end{cases}
\]
while the kernel appearing in the integral term is given by
\[
k_i(t,s)=
\begin{cases}
    \hat{k}_i(t,s), & \text{if } (t,s)\in[0,1]\times[0,1],\\
    0,& \text{if } (t,s)\in([-\sigma_i,\tau_i]\setminus(0,1))\times [0,1],
\end{cases}
\]
where $\hat{k}_i$ is the Green's function associated to the BVP
\begin{equation*}
\left\{\begin{array}{rll}
-u_{i}''(t) &= h_{i}(t), \quad & t \in [0, 1], \\
u_{i}(0) &= \psi_{i}(0), \\
u_{i}(1) &= \varphi_{i}(1).
\end{array} \right.
\end{equation*}

\begin{defn} \label{def_solution}
By a solution of the functional BVP \eqref{sys_diff} we mean a solution
$(u_1,u_2)\in X_1\times X_2$
of the integral system \eqref{sys_integral}.
\end{defn}

In the Banach space $X_1$ we consider the cone $K_1$ of nonnegative functions. 
Let, for $i=1,2$, $\rho_{i}>0$ be given and assume that the operator $H_i$ is defined (at least) on
$\overline{K_{1, \rho_{1}}^{\phi_{1}}} \times \overline{B_{2, \rho_{2}}^{\phi_{2}}}$.
Let
$F_{i}: \overline{K_{1, \rho_{1}}^{\phi_{1}}} \times \overline{B_{2, \rho_{2}}^{\phi_{2}}} \rightarrow X_i$
be defined by
\begin{equation}
\label{def_int.op}
F_i(u_1,u_2)(t) = \int_{0}^{1} k_{i}(t,s) f_{i}(s, u_{1}(s), u_{2}(s), H_{i}[u_{1},u_{2}](s)) \,ds,
\ t\in[-\sigma_i,\tau_i],
\end{equation}
so that the system~\eqref{sys_integral} takes the form
\begin{equation*}
\begin{cases}
    u_1(t) = \phi_1(t) + \lambda_1 F_1(u_1,u_2)(t),&t\in[-\sigma_1,\tau_1],\\
    u_2(t) = \phi_2(t) + \lambda_2 F_2(u_1,u_2)(t),&t\in[-\sigma_2,\tau_2],
\end{cases}
\end{equation*}
which in turn can be written as a fixed-point system in the product space $X_1 \times X_2$, namely,
\begin{equation}
\begin{cases}\label{sys_abstract}
    u_1 = \phi_1  + \lambda_1 F_1(u_1,u_2),\\
    u_2 = \phi_2  + \lambda_2 F_2(u_1,u_2).
\end{cases}
\end{equation}
With the above ingredients, we can state the following existence result.

\begin{thm}\label{ThmG}
Let $\rho_1,\rho_2>0$ be as above and assume that the following conditions hold~$(i=1,2)$:
\begin{enumerate}
        
    \item $H_i:\overline{K_{1, \rho_{1}}^{\phi_{1}}} \times \overline{B_{2, \rho_{2}}^{\phi_{2}}} \to C[0,1]$ is continuous and bounded by two numbers $\underline{H_{i,\rho_1,\rho_2}}, \overline{H_{i,\rho_1,\rho_2}}$, that is
    \[
    \underline{H_{i,\rho_1,\rho_2}} \leq H_i(u_1,u_2)(t) \leq \overline{H_{i,\rho_1,\rho_2}} \quad \text{for all } (t,u_1,u_2) \in [0,1]\times \overline{K_{1, \rho_{1}}^{\phi_{1}}} \times \overline{B_{2, \rho_{2}}^{\phi_{2}}}.
    \]
    
    \item $f_1: \Pi_{1,\rho_1,\rho_2} \subset \R^4 \to \R_+$ and $f_2: \Pi_{2,\rho_1,\rho_2} \subset \R^4 \to \R$ are continuous, where
    \[
    \Pi_{i,\rho_1,\rho_2} = [0,1] \times \left[\min_{[-\sigma_1,\tau_1]}\phi_1,\max_{[-\sigma_1,\tau_1]}\phi_1+\rho_1\right] \times \left[\min_{[-\sigma_2,\tau_2]}\phi_2-\rho_2,\max_{[-\sigma_2,\tau_2]}\phi_2+\rho_2\right] \times [\underline{H_{i,\rho_1,\rho_2}},\overline{H_{i,\rho_1,\rho_2}}].
    \]
    
    \item There exist two continuous functions $\overline{f_{i,\rho_1,\rho_2}},\underline{f_{i,\rho_1,\rho_2}}:[0,1]\to\R$ such that
    \[
    \underline{f_{i,\rho_1,\rho_2}}(t) \leq f_i(t,u,v,w) \leq \overline{f_{i,\rho_1,\rho_2}}(t) \quad \text{for all } \big(t,u,v,w\big) \in \Pi_{i,\rho_1,\rho_2}.
    \]
    
    \item Define, for $t\in[-\sigma_i,\tau_i]$,
    \[
    \underline{F_{i,\rho_1,\rho_2}}(t) = \int_0^1 k_i(t,s)\underline{f_{i,\rho_1,\rho_2}}(s)\,ds \quad\text{and}\quad \overline{F_{i,\rho_1,\rho_2}}(t) = \int_0^1 k_i(t,s)\overline{f_{i,\rho_1,\rho_2}}(s)\,ds
    \]
    and assume that at least one between $\ref{a}-\ref{b}$ and at least one between $\ref{c}-\ref{d}$ hold, where:
    \begin{enumerate}[label=(4\alph*)]
        \item \label{a} There exists $t_{1,\rho_1,\rho_2} \in [-\sigma_1, \tau_1]$ such that $\overline{F_{1,\rho_1,\rho_2}}(t_{1,\rho_1,\rho_2})<0$;\vspace{1mm}
        \item \label{b} There exists $t_{1,\rho_1,\rho_2} \in [-\sigma_1, \tau_1]$ such that $\underline{F_{1,\rho_1,\rho_2}}(t_{1,\rho_1,\rho_2})>0$;\vspace{1mm}
        \item \label{c} There exists $t_{2,\rho_1,\rho_2} \in [-\sigma_2, \tau_2]$ such that $\overline{F_{2,\rho_1,\rho_2}}(t_{2,\rho_1,\rho_2})<0$;\vspace{1mm}
        \item \label{d} There exists $t_{2,\rho_1,\rho_2} \in [-\sigma_2, \tau_2]$ such that $\underline{F_{2,\rho_1,\rho_2}}(t_{2,\rho_1,\rho_2})>0$.\vspace{1mm}
    \end{enumerate}
\end{enumerate}
Then there exist
$\lambda_{1,1},\lambda_{2,1},\lambda_{1,2} >0$, $\lambda_{2,2}<0$ and $(u_{1,j},u_{2,j})\in \overline{K_{1, \rho_{1}}^{\phi_{1}}} \times \overline{B_{2, \rho_{2}}^{\phi_{2}}} $, $j=1,2$, that satisfy the system~\eqref{sys_diff}.
More precisely,
\begin{gather*}
    (u_{1,1},u_{2,1})=(\phi_1+v_{1,1},\phi_2+v_{2,1}), \text{ with } \| v_{1,1}\|_1=\rho_1, \| v_{2,1}\|_2=\rho_2,\\
    (u_{1,2},u_{2,2})=(\phi_1+v_{1,2},\phi_2+v_{2,2}), \text{ with } \| v_{1,2}\|_1=\rho_1, \| v_{2,2}\|_2=\rho_2
\end{gather*}
are such that, for $j=1,2$,
\[
\begin{cases}
    u_{1,j}(t) = \phi_1(t) + \lambda_{1,j} \int_0^1 k_1(t,s) f_1(s,u_{1,j}(s),u_{2,j}(s),H_1[u_{1,j},u_{2,j}](s)) \,ds,\ t \in [-\sigma_1,\tau_1], \\
    u_{2,j}(t) = \phi_2(t) + \lambda_{2,j} \int_0^1 k_2(t,s) f_2(s,u_{1,j}(s),u_{2,j}(s),H_2[u_{1,j},u_{2,j}](s)) \,ds,\ t \in [-\sigma_2,\tau_2]. \\
\end{cases}
\]
Moreover, the following implications hold:
    \begin{align*}
        \ref{a} &\quad \Rightarrow \quad |\lambda_{1,j}| \leq -\rho_1 / \overline{F_{1,\rho_1,\rho_2}}(t_{1,\rho_1,\rho_2});\\
        \ref{b} &\quad \Rightarrow \quad |\lambda_{1,j}| \leq \rho_1 / \underline{F_{1,\rho_1,\rho_2}}(t_{1,\rho_1,\rho_2});\\
        \ref{c} &\quad \Rightarrow \quad |\lambda_{2,j}| \leq -\rho_2 / \overline{F_{2,\rho_1,\rho_2}}(t_{2,\rho_1,\rho_2});\\
        \ref{d} &\quad \Rightarrow \quad |\lambda_{2,j}| \leq \rho_2 / \underline{F_{2,\rho_1,\rho_2}}(t_{2,\rho_1,\rho_2}).
    \end{align*}
Finally, we also have that, for $i,j \in \{1,2\}$,
\[
|\lambda_{i,j}| \geq \rho_i / \max\{\| \underline{F_{i,\rho_1,\rho_2}} \|_i, \| \overline{F_{i,\rho_1,\rho_2}} \|_i\}.
\]
\end{thm}

\begin{proof}
Let $\rho_1,\rho_2>0$ be fixed and, for $i=1,2$, let $F_{i}: \overline{K_{1, \rho_{1}}^{\phi_{1}}} \times \overline{B_{2, \rho_{2}}^{\phi_{2}}} \rightarrow X_i$
be as in 
\eqref{def_int.op}. Note that $F_i$ is well-defined and
 compact because of assumptions $(1)-(3)$;
thus, so is $F:=(F_1,F_2)$. Moreover, 
$F_1$ takes values in the cone $K_1$. Now,
by $(4)$ we get
\[
\underline{F_{i,\rho_1,\rho_2}}(t) \leq F_{i}(u_1,u_2)(t) \leq \overline{F_{i,\rho_1,\rho_2}}(t) \quad\text{for every } (t,u_1,u_2) \in [-\sigma_i,\tau_i] \times \overline{K_{1, \rho_{1}}^{\phi_{1}}} \times \overline{B_{2, \rho_{2}}^{\phi_{2}}}.
\]
Now, given
$(u_1,u_2) \in \partial{K_{1, \rho_{1}}^{\phi_{1}}} \times \partial{B_{2, \rho_{2}}^{\phi_{2}}}$,
we obtain following estimates.
\begin{itemize}
\item If $\ref{a}$ holds, then
    \[
    \|F_1(u_1,u_2)\|_1 \geq |F_1(u_1,u_2)(t_{1,\rho_1,\rho_2})| = -F_1(u_1,u_2)(t_{1,\rho_1,\rho_2}) \geq -\overline{F_{1,\rho_1,\rho_2}}(t_{1,\rho_1,\rho_2}) > 0,
    \]
    which yields
    \[
    \inf_{ (u_1,u_2) \in \partial{K_{1, \rho_{1}}^{\phi_{1}}} \times \partial{B_{2, \rho_{2}}^{\phi_{2}}}} \| F_1(u_1,u_2) \|_1 > 0.
    \]

\item If $\ref{b}$ holds, then
    \[
    \|F_1(u_1,u_2)\|_1 \geq |F_1(u_1,u_2)(t_{1,\rho_1,\rho_2})| = F_1(u_1,u_2)(t_{1,\rho_1,\rho_2}) \geq \underline{F_{1,\rho_1,\rho_2}}(t_{1,\rho_1,\rho_2}) > 0,
    \]
    which yields
    \[
    \inf_{ (u_1,u_2) \in \partial{K_{1, \rho_{1}}^{\phi_{1}}} \times \partial{B_{2, \rho_{2}}^{\phi_{2}}}} \| F_1(u_1,u_2) \|_1 > 0.
    \]

\item If $\ref{c}$ holds, then
    \[
    \|F_2(u_1,u_2)\|_2 \geq |F_2(u_1,u_2)(t_{2,\rho_1,\rho_2})| = -F_2(u_1,u_2)(t_{2,\rho_1,\rho_2}) \geq -\overline{F_{2,\rho_1,\rho_2}}(t_{2,\rho_1,\rho_2}) > 0,
    \]
    which yields
    \[
    \inf_{ (u_1,u_2) \in \partial{K_{1, \rho_{1}}^{\phi_{1}}} \times \partial{B_{2, \rho_{2}}^{\phi_{2}}}} \| F_2(u_1,u_2) \|_2 > 0.
    \]

\item If $\ref{d}$ holds, then
    \[
    \|F_2(u_1,u_2)\|_2 \geq |F_2(u_1,u_2)(t_{2,\rho_1,\rho_2})| = F_2(u_1,u_2)(t_{2,\rho_1,\rho_2}) \geq \underline{F_{2,\rho_1,\rho_2}}(t_{2,\rho_1,\rho_2}) > 0,
    \]
    which yields
    \[
    \inf_{ (u_1,u_2) \in \partial{K_{1, \rho_{1}}^{\phi_{1}}} \times \partial{B_{2, \rho_{2}}^{\phi_{2}}}} \| F_2(u_1,u_2) \|_2 > 0.
    \]
\end{itemize}

In summary, every hypothesis of Theorem~\ref{th_sys_kb} is satisfied; hence there exist two solution pairs $(u_{1,j},u_{2,j})$, $j=1,2$, of the abstract system \eqref{sys_abstract}, corresponding, respectively, to the couples of parameters $(\lambda_{1,1},\lambda_{2,1}) \in (0,+\infty) \times (0,+\infty)$ and $(\lambda_{1,2},\lambda_{2,2}) \in (0,+\infty)\times(-\infty,0)$.
Thus, the first part of the assertion follows.

Now, to complete the proof, let us fix $(\lambda_1,\lambda_2) \in \{(\lambda_{1,1},\lambda_{2,1}),
(\lambda_{1,2},\lambda_{2,2})\}$
and consider the associated couple of functions $(u_1,u_2)=(\phi_1+v_1,\phi_2+v_2)$. For $i=1,2$:
\begin{itemize}
\item If $\ref{a}$ holds, then we have that
\begin{align*}
    &u_1(t) = \phi_1(t) + \lambda_1 F_1(u_1,u_2)(t) \quad\text{for all } t \in [-\sigma_1,\tau_1] \\
    \Rightarrow \ &u_1(t) - \phi_1(t) = \lambda_1 F_1(u_1,u_2)(t) \quad\text{for all } t \in [-\sigma_1,\tau_1] \\
    \Rightarrow \ & \rho_1=\|u_1-\phi_1\|_1 = |\lambda_1| \|F_1(u_1,u_2)\|_1 \geq -|\lambda_1|\overline{F_{1,\rho_1,\rho_2}}(t_{1,\rho_1,\rho_2}) > 0 \\
    \Rightarrow \ & |\lambda_1| \leq -\rho_1/\overline{F_{1,\rho_1,\rho_2}}(t_{1,\rho_1,\rho_2}).
\end{align*}

\item If $\ref{b}$ holds, then we have that
\begin{align*}
    &u_1(t) = \phi_1(t) + \lambda_1 F_1(u_1,u_2)(t) \quad\text{for all } t \in [-\sigma_1,\tau_1] \\
    \Rightarrow \ &u_1(t) - \phi_1(t) = \lambda_1 F_1(u_1,u_2)(t) \quad\text{for all } t \in [-\sigma_1,\tau_1] \\
    \Rightarrow \ & \rho_1=\|u_1-\phi_1\|_1 = |\lambda_1| \|F_1(u_1,u_2)\|_1 \geq |\lambda_1|\underline{F_{1,\rho_1,\rho_2}}(t_{1,\rho_1,\rho_2}) > 0 \\
    \Rightarrow \ & |\lambda_1| \leq \rho_1/\underline{F_{1,\rho_1,\rho_2}}(t_{1,\rho_1,\rho_2}).
\end{align*}

\item If $\ref{c}$ holds, then we have that
\begin{align*}
    &u_2(t) = \phi_2(t) + \lambda_2 F_2(u_1,u_2)(t) \quad\text{for all } t \in [-\sigma_2,\tau_2] \\
    \Rightarrow \ &u_2(t) - \phi_2(t) = \lambda_2 F_2(u_1,u_2)(t) \quad\text{for all } t \in [-\sigma_2,\tau_2] \\
    \Rightarrow \ & \rho_2=\|u_2-\phi_2\|_2 = |\lambda_2| \|F_2(u_1,u_2)\|_2 \geq -|\lambda_2|\overline{F_{2,\rho_1,\rho_2}}(t_{2,\rho_1,\rho_2}) > 0 \\
    \Rightarrow \ & |\lambda_2| \leq -\rho_2/\overline{F_{2,\rho_1,\rho_2}}(t_{2,\rho_1,\rho_2}).
\end{align*}

\item If $\ref{d}$ holds, then we have that
\begin{align*}
    &u_2(t) = \phi_2(t) + \lambda_2 F_2(u_1,u_2)(t) \quad\text{for all } t \in [-\sigma_2,\tau_2] \\
    \Rightarrow \ &u_2(t) - \phi_2(t) = \lambda_2 F_2(u_1,u_2)(t) \quad\text{for all } t \in [-\sigma_2,\tau_2] \\
    \Rightarrow \ & \rho_2=\|u_2-\phi_2\|_2 = |\lambda_2| \|F_2(u_1,u_2)\|_2 \geq |\lambda_2|\underline{F_{2,\rho_1,\rho_2}}(t_{2,\rho_1,\rho_2}) > 0 \\
    \Rightarrow \ & |\lambda_2| \leq \rho_2/\underline{F_{2,\rho_1,\rho_2}}(t_{2,\rho_1,\rho_2}).
\end{align*}
\end{itemize}
Finally, recall that if $a,b,c \in \R$ are such that $a\leq b \leq c$, then we have that
\[
|b| \leq \max\{|a|,|c|\}.
\]
This means that, for $i=1,2$ and for every $t \in [-\sigma_i,\tau_i]$,
\begin{align*}
|u_i(t)-\phi_i(t)| &= |\lambda_i| |F_{i,\rho_1,\rho_2}(u_1,u_2)(t)|\\
&\leq |\lambda_i| \max\{|\underline{F_{i,\rho_1,\rho_2}}(t)|,|\overline{F_{i,\rho_1,\rho_2}}(t)|\}\\
&\leq |\lambda_i| \max\{\|\underline{F_{i,\rho_1,\rho_2}}\|_i,\|\overline{F_{i,\rho_1,\rho_2}}\|_i\}.
\end{align*}
This yields
\[
\rho_i=\|u_i-\phi_i\|_i \leq |\lambda_i| \max\{\|\underline{F_{i,\rho_1,\rho_2}}\|_i,\|\overline{F_{i,\rho_1,\rho_2}}\|_i\}.
\]
If one between \ref{a}--\ref{d} holds, we have that $\max\{\|\underline{F_{i,\rho_1,\rho_2}}\|_i,\|\overline{F_{i,\rho_1,\rho_2}}\|_i\} > 0$, hence we have that
\[
|\lambda_i| \geq \rho_i/\max\{\|\underline{F_{i,\rho_1,\rho_2}}\|_i,\|\overline{F_{i,\rho_1,\rho_2}}\|_i\}.
\qedhere
\]
\end{proof}

\begin{rem}
    Note that if $\ref{b}$ or $\ref{d}$ hold, then we can obtain a better estimate as follows:
    \begin{itemize}
        \item If $\ref{b}$ holds, then we must have $\max_{[-\sigma_1,\tau_1]} \underline{F_{1,\rho_1,\rho_2}} > 0$, therefore we can consider $t_{1,\rho_1,\rho_2} \in \operatorname{argmax}(F_{1,\rho_1,\rho_2})$ to get the optimal estimate
        \[
        |\lambda_1| \leq \rho_1/\max_{[-\sigma_1,\tau_1]} \underline{F_{1,\rho_1,\rho_2}}.
        \]

        \item If $\ref{d}$ holds, then we must have $\max_{[-\sigma_2,\tau_2]} \underline{F_{2,\rho_1,\rho_2}} > 0$, therefore we can consider $t_{2,\rho_1,\rho_2} \in \operatorname{argmax}(F_{2,\rho_1,\rho_2})$ to get the optimal estimate
        \[
        |\lambda_2| \leq \rho_2/\max_{[-\sigma_2,\tau_2]} \underline{F_{2,\rho_1,\rho_2}}.
        \]
        
    \end{itemize}
\end{rem}

Now, we propose an example to show the applicability of the latter theorem to systems of two functional differential equations.

\begin{ex}
Let us consider the system
\begin{equation}\label{sys_example}
\begin{cases}
    -u_1''(t)=\lambda_1 f_1(t,u_1(t),u_2(t),H_1(u_1,u_2)(t)), &t \in (0,1), \\
    -u_2''(t)=\lambda_2 f_2(t,u_1(t),u_2(t),H_2(u_1,u_2)(t)), &t \in (0,1), \\
    u_1(t) = \frac{1}{3}t^2, &t \in \left[-\frac{\pi}{8},0\right], \\
    u_2(t) = -\frac{1}{3}t^3, &t \in \left[-\frac{\pi}{6},0\right], \\
    u_1(t) = \frac{1}{12}\sqrt{t-1}, &t \in \left[1,1+\frac{\pi}{8}\right], \\
    u_2(t) = -\frac{1}{3}(t-1)^2, &t \in \left[1,1+\frac{\pi}{16}\right].
\end{cases}
\end{equation}
This is a functional advanced-delayed differential system, that is a particular case of system~\eqref{sys_diff}, in the product space $X_{1} \times X_{2}$,
where $X_1=C[-\frac{\pi}{8}, 1+\frac{\pi}{8}]$ and $X_2=C[-\frac{\pi}{6}, 1+\frac{\pi}{16}]$.
The maps
$f_1,f_2: \R^4 \to \R$ are defined by
\begin{align*}
f_1(t,u,v,w) &= \cos(t)e^{u+w},\\
f_2(t,u,v,w) &= \sin \left( \frac{9}{5} \pi t \right)e^{v+w},
\end{align*}
while the operators $H_1,H_2$ are defined by setting, for every for every $t \in [0,1]$,
\begin{align*}
H_1(u_1,u_2)(t) &= u_1\Big(t + \frac{3}{10}(1-t)\Big)+ \sin\left(u_2(t) + u_2\Big(t - \frac{1}{4}\arctan \|u_2\|\Big)\right),\\
H_2(u_1,u_2)(t) &= -u_2\Big(t - \frac{1}{10}\sin(2\pi t)\Big) + \frac{1}{2}\cos\Bigg(u_1(t) - u_1\Big(t + \frac{1}{8}\arctan \|u_1\|\Big)\Bigg).
\end{align*}
Note that system \eqref{sys_example} is of the form~\eqref{sys_diff}
and that state-dependent terms occur in the operators $H_1$ and $H_2$.

In the Banach space $X_1$ we consider the cone $K_1$ of nonnegative functions and we observe that
system \eqref{sys_example} may be rewritten into an equivalent system of integral equations according to Definition \ref{def_solution}.
In particular, we have that
\[
\phi_1(t)=
\begin{cases}
\frac{1}{3}t^2, &t \in \left[-\frac{\pi}{8},0\right], \\
0, &t \in [0,1], \\
\frac{1}{12}\sqrt{t-1}, &t \in \left[1,1+\frac{\pi}{8}\right],
\end{cases}
\]
and
\[
\phi_2(t)=
\begin{cases}
-\frac{1}{3}t^3, &t \in \left[-\frac{\pi}{6},0\right], \\
0, &t \in [0,1], \\
-\frac{1}{3}(t-1)^2, &t \in \left[1,1+\frac{\pi}{16}\right].
\end{cases}
\]

Now, to apply Theorem \ref{ThmG} we need to provide some suitable estimates.
Let $\rho_1,\rho_2>0$ be given and assume that $\|u_i-\phi_i\|_i=\rho_i$,
$i=1,2$. Adopting the notation of Theorem \ref{ThmG}, note that we may choose
\begin{align*}
\underline{H_{1,\rho_1,\rho_2}} &= -1,\quad \overline{H_{1,\rho_1,\rho_2}} = \rho_1+1,\\
\underline{H_{2,\rho_1,\rho_2}} &= -\rho_2-\frac{1}{2},\quad \overline{H_{2,\rho_1,\rho_2}} = \rho_2+\frac{1}{2}.
\end{align*}
Now, set
\[
g(t)= \sin\left(\frac{9}{5}\pi t\right),
\]
and call its positive and negative parts $g_+(t)$ and $g_-(t)$ respectively; then we may consider
\begin{align*}
\underline{f_{1,\rho_1,\rho_2}}(t) &= e^{m_1-1}\cos(t) = e^{-1}\cos(t),\\
\overline{f_{1,\rho_1,\rho_2}}(t) &= e^{M_1+1+2\rho_1}\cos(t),\\
\underline{f_{2,\rho_1,\rho_2}}(t) &= e^{m_2-2\rho_2-\frac{1}{2}}g_+(t)-e^{M_2+2\rho_2+\frac{1}{2}}g_-(t),\\ 
\overline{f_{2,\rho_1,\rho_2}}(t) &= e^{M_2+2\rho_2+\frac{1}{2}}g_+(t)-e^{m_2-2\rho_2-\frac{1}{2}}g_-(t),
\end{align*}
where
\begin{align*}
m_1=\min_{[-\frac{\pi}{8},1+\frac{\pi}{8}]}\phi_1=0,\qquad M_1=\max_{[-\frac{\pi}{8},1+\frac{\pi}{8}]}\phi_1=\frac{1}{12}\sqrt{\frac{\pi}{8}},\\
m_2=\min_{[-\frac{\pi}{6},1+\frac{\pi}{16}]}\phi_2=-\frac{\pi^2}{768},\qquad M_2=\max_{[-\frac{\pi}{6},1+\frac{\pi}{16}]}\phi_2=\frac{\pi^3}{648}.
\end{align*}

With the above choice of $\underline{f_{i,\rho_1,\rho_2}}$, $\overline{f_{i,\rho_1,\rho_2}}$ we have, as in condition (4) of  Theorem \ref{ThmG}, the corresponding maps $\underline{F_{i,\rho_1,\rho_2}}$, $\overline{F_{i,\rho_1,\rho_2}}$. Namely,
for every $t \in \left[ -\frac{\pi}{8}, 1+\frac{\pi}{8} \right]$,
\[
\underline{F_{1,\rho_1,\rho_2}}(t) = e^{-1}\int_0^1 k_1(t,s)\cos(s)\,ds,
\]
and
\[
\overline{F_{1,\rho_1,\rho_2}}(t) = e^{M_1+2\rho_1+1}\int_0^1 k_1(t,s)\cos(s)\,ds,
\]
and also, for every $t \in \left[ -\frac{\pi}{6}, 1+\frac{\pi}{16} \right]$,
\[
\underline{F_{2,\rho_1,\rho_2}}(t) = e^{m_2-2\rho_2-\frac{1}{2}} \int_0^{5/9} k_2(t,s)\sin(bs)\,ds + e^{M_2+2\rho_2+\frac{1}{2}} \int_{5/9}^1 k_2(t,s)\sin(bs)\,ds
\]
and
\[
\overline{F_{2,\rho_1,\rho_2}}(t) = e^{M_2+2\rho_2+\frac{1}{2}} \int_0^{5/9} k_2(t,s)\sin(bs)\,ds + e^{m_2-2\rho_2-\frac{1}{2}} \int_{5/9}^1 k_2(t,s)\sin(bs)\,ds,
\]
where $b=9\pi/5$,
\[
k_1(t,s)=
\begin{cases}
    t(1-s),&\text{if } 0\leq t\leq s \leq 1,\\
    s(1-t),&\text{if } 0\leq s\leq t \leq 1,\\
    0,&\text{if } t \in \left[ -\frac{\pi}{8}, 1+\frac{\pi}{8} \right]\setminus(0,1).
\end{cases}
\]
and
\[
k_2(t,s)=
\begin{cases}
    t(1-s),&\text{if } 0\leq t\leq s \leq 1,\\
    s(1-t),&\text{if } 0\leq s\leq t \leq 1,\\
    0,&\text{if } t \in \left[ -\frac{\pi}{6}, 1+\frac{\pi}{16} \right]\setminus(0,1).
\end{cases}
\]
Firstly note that $\underline{F_{1,\rho_1,\rho_2}}$ and $\overline{F_{1,\rho_1,\rho_2}}$ do not depend on $\rho_2$, while $\underline{F_{2,\rho_1,\rho_2}}$ and $\overline{F_{2,\rho_1,\rho_2}}$ do not depend on $\rho_1$. Moreover, since $\underline{F_{1,\rho_1,\rho_2}},\underline{F_{2,\rho_1,\rho_2}}, \overline{F_{1,\rho_1,\rho_2}}$ and $\overline{F_{2,\rho_1,\rho_2}}$ are zero outside the interval $[0,1]$, to apply Theorem~\ref{ThmG} we can focus on their restrictions on $[0,1]$, where both $k_1$ and $k_2$ are equal to
\[
k(t,s)=
\begin{cases}
    t(1-s),&\text{if } 0\leq t\leq s \leq 1,\\
    s(1-t),&\text{if } 0\leq s\leq t \leq 1.
\end{cases}
\]
First of all, a direct computation yields, for every $t \in [0,1]$,
\[
\int_0^1 k(t,s)\cos(s)\,ds = \cos(t)+t(1-\cos(1))-1,
\]
whose maximum point is $t_0=\arcsin(1-\cos(1)) \in [0,1]$. Hence,
\[
\underline{F_{1,\rho_1,\rho_2}}(t_0) \geq e^{-1}[\cos(t_0)+t_0(1-\cos(1))-1]>0,
\]
which implies that for every $\rho_1,\rho_2>0$ there exists $t_{1,\rho_1,\rho_2}=t_0 \in [0,1]$ such that
\[
\underline{F_{1,\rho_1,\rho_2}}(t_{1,\rho_1,\rho_2})>0.
\]
The graphs of the functions $\underline{F_{1,\rho_1,\rho_2}}, \overline{F_{1,\rho_1,\rho_2}}$ are illustrated in Figure~\ref{fig:F1 upp and low}.
\begin{figure}[h]
    \centering
    \includegraphics[width=0.62\linewidth]{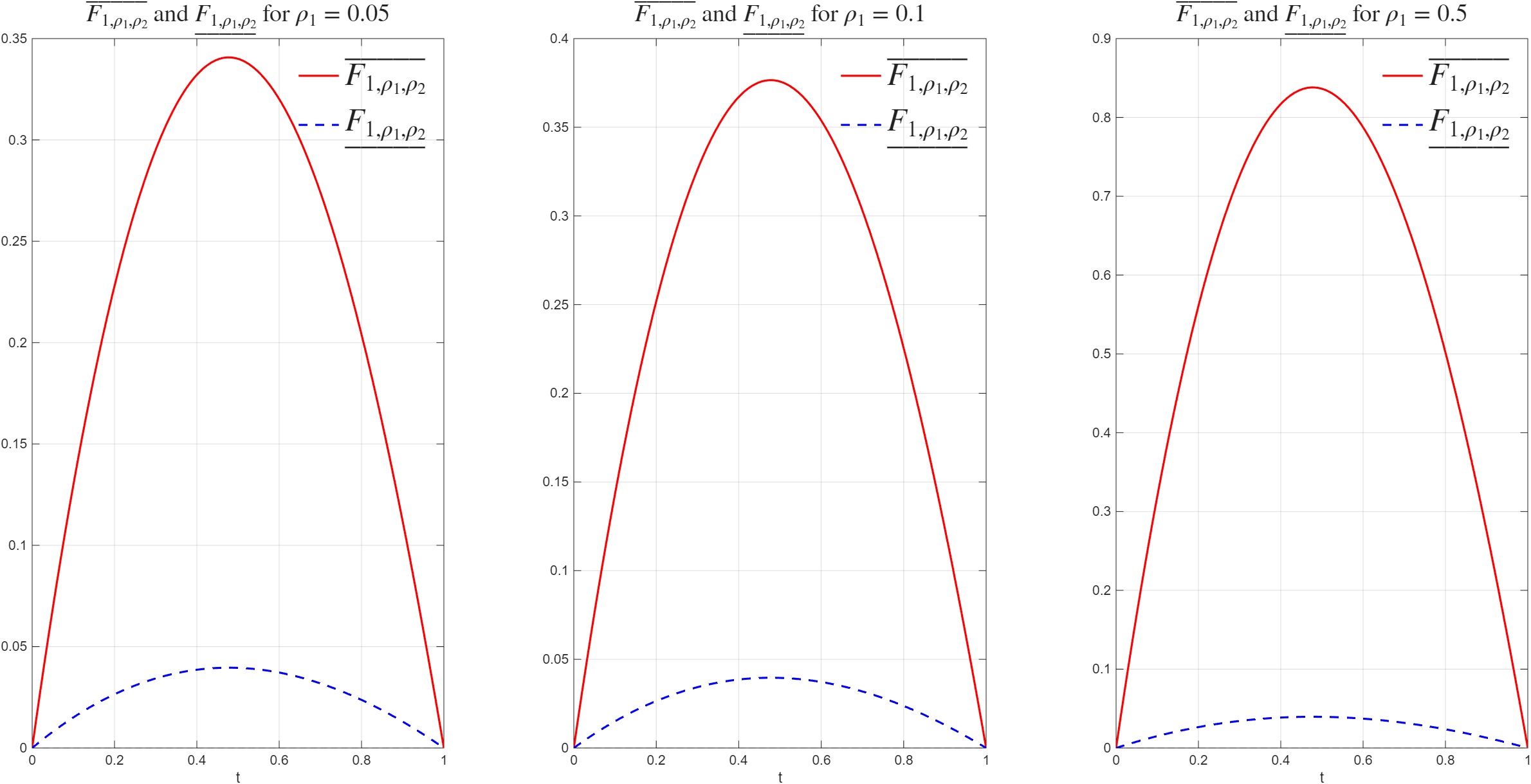}
    \caption{Plot of the functions $\overline{F_{1,\rho_1,\rho_2}}$ and $\underline{F_{1,\rho_1,\rho_2}}$} for $\rho_1=0.05$, $\rho_1=0.1$ and $\rho_1=0.5$.
    \label{fig:F1 upp and low}
\end{figure}\\
Also by a direct computation, for every $t \in [0,1]$, we have
\[
\int_{0}^{5/9} k(t, s) \sin(bs) \, ds = 
\begin{cases} 
\dfrac{25 \sin(bt)}{81 \pi^2} + \dfrac{20t}{81 \pi}, & \text{if } t \le 5/9, \vspace{10pt}\\
\dfrac{25(1-t)}{81 \pi}, & \text{if } t > 5/9,
\end{cases}
\]
and
\[
\int_{5/9}^{1} k(t, s) \sin(bs) \, ds = 
\begin{cases} 
-\dfrac{25 \sin(b)}{81 \pi^2}t - \dfrac{20t}{81 \pi}, & \text{if } t \le 5/9, \\[10pt]
\dfrac{25 \sin(bt)}{81 \pi^2} - \dfrac{25(1-t)}{81 \pi} - \dfrac{25 \sin(b)}{81 \pi^2}t, & \text{if } t > 5/9.
\end{cases}
\]
In summary, for every $t \in \left[ -\frac{\pi}{6}, 1+\frac{\pi}{16} \right]$, $\overline{F_{2,\rho_1,\rho_2}}(t)$ and $\underline{F_{2,\rho_1,\rho_2}}(t)$ read as:
\[
\overline{F_{2,\rho_1,\rho_2}}=
\begin{cases}
\dfrac{5e^{M_2+2\rho_2+\frac{1}{2}}}{81\pi} \left(\dfrac{5 \sin(bt)}{\pi} + 4t\right) - \dfrac{5e^{m_2-2\rho_2-\frac{1}{2}}}{81\pi}\left(\dfrac{5 \sin(b)}{\pi} + 4\right)t,
&\text{if } t\in[0,5/9],\\

\dfrac{25e^{M_2+2\rho_2+\frac{1}{2}}}{81\pi}(1-t) + \dfrac{25e^{m_2-2\rho_2-\frac{1}{2}}}{81\pi} \left(\dfrac{\sin(bt)}{\pi} +t-1 - \dfrac{ \sin(b)}{\pi}t\right),
&\text{if } t \in (5/9,1],\\

0, &\text{if } t \notin [0,1].\\
\end{cases}
\]
and
\[
\underline{F_{2,\rho_1,\rho_2}}=
\begin{cases}
\dfrac{5e^{m_2-2\rho_2-\frac{1}{2}}}{81\pi} \left(\dfrac{5 \sin(bt)}{\pi} + 4t\right) - \dfrac{5e^{M_2+2\rho_2+\frac{1}{2}}}{81\pi}\left(\dfrac{5 \sin(b)}{\pi} + 4\right)t,
&\text{if } t\in[0,5/9],\\

\dfrac{25e^{m_2-2\rho_2-\frac{1}{2}}}{81\pi}(1-t) + \dfrac{25e^{M_2+2\rho_2+\frac{1}{2}}}{81\pi} \left(\dfrac{\sin(bt)}{\pi} +t-1 - \dfrac{ \sin(b)}{\pi}t\right), 
&\text{if } t \in (5/9,1],\\

0, &\text{if } t \notin [0,1].\\
\end{cases}
\]
\begin{figure}[h]
    \centering
    \includegraphics[width=0.62\linewidth]{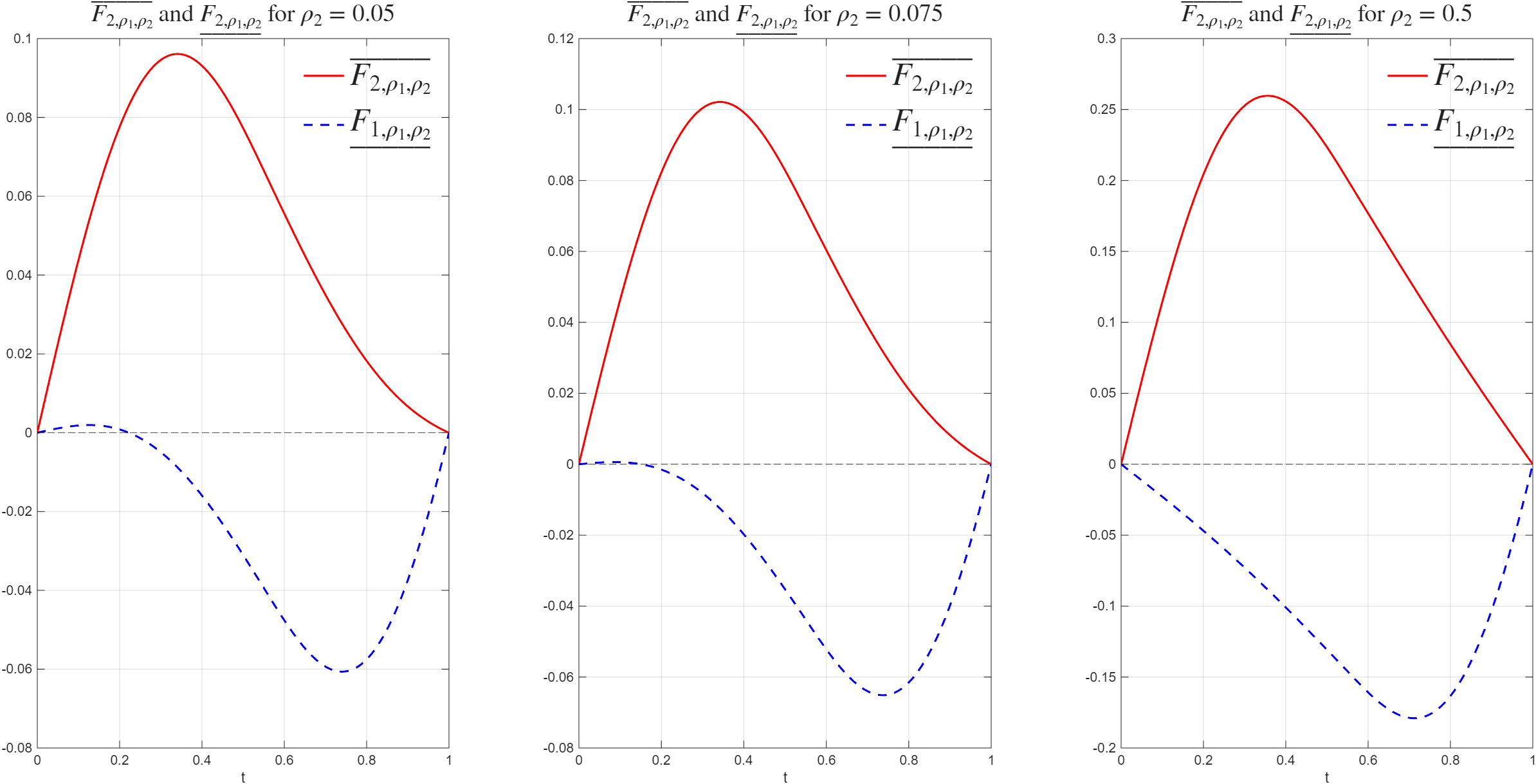}
    \caption{Plot of the functions $\overline{F_{2,\rho_1,\rho_2}}$ and $\underline{F_{2,\rho_1,\rho_2}}$} for $\rho_2=0.05$, $\rho_2=0.075$ and $\rho_2=0.5$.
    \label{fig:F2 upp and low}
\end{figure}\\
Note that $\underline{F_{2,\rho_1,\rho_2}}$ and $\overline{F_{2,\rho_1,\rho_2}}$ are $C^1([0,1])$ and they are zero at $t=0$ and $t=1$.\\
Now, we search for the absolute minimum points in $[0,1]$ of $\overline{F_{2,\rho_1,\rho_2}}$.

Let us start with the interval $[0,5/9]$:
\begin{align*}
\overline{F_{2,\rho_1,\rho_2}}'(t) \geq 0 \quad&\iff\quad
e^{M_2+2\rho_2+\frac{1}{2}} \left(9\cos(bt) + 4\right) - e^{m_2-2\rho_2-\frac{1}{2}}\left(\dfrac{5 \sin(b)}{\pi} + 4\right) \geq 0 \\ &\iff\quad
e^{M_2+2\rho_2+\frac{1}{2}} \left(9\cos(bt) + 4\right) \geq e^{m_2-2\rho_2-\frac{1}{2}}\left(\dfrac{5 \sin(b)}{\pi} + 4\right) \\ &\iff\quad
\cos(bt) \geq \frac{e^{m_2-M_2-4\rho_2-1}}{9}\left(\dfrac{5 \sin(b)}{\pi} + 4\right) - \frac{4}{9} =: \gamma_{\rho_2}.
\end{align*}
The last inequality is solvable if and only if $\gamma_{\rho_2} \in [-1,1]$. Actually, this is always true for every $\rho_2>0$. Indeed,
\begin{align*}
\frac{e^{m_2-M_2-4\rho_2-1}}{9}\left(\dfrac{5 \sin(b)}{\pi} + 4\right) - \frac{4}{9} < 1 \quad&\iff\quad e^{m_2-M_2-4\rho_2-1}\left(\dfrac{5 \sin(b)}{\pi} + 4\right) < 13.
\end{align*}
Since the (LHS) of the latter inequality is smaller than $\dfrac{5 \sin(b)}{\pi} + 4 \approx 3.06$, $\gamma_{\rho_2}<1$ for every $\rho_2>0$. At the same time,
\begin{align*}
\frac{e^{m_2-M_2-4\rho_2-1}}{9}\left(\dfrac{5 \sin(b)}{\pi} + 4\right) - \frac{4}{9} > -1 \quad&\iff\quad e^{m_2-M_2-4\rho_2-1}\left(\dfrac{5 \sin(b)}{\pi} + 4\right) > -5.
\end{align*}
Since the (LHS) of the last inequality is always positive, it is always true, hence $\gamma_{\rho_2}>-1$ for every $\rho_2>0$.
This means that $\overline{F_{2,\rho_1,\rho_2}}$ is increasing for $t\leq \frac{5}{9\pi}\arccos(\gamma_{\rho_2}) \in [0,5/9]$ and decreasing for $t \geq \frac{5}{9\pi}\arccos(\gamma_{\rho_2})$. Therefore, its minimum point in $[0,5/9]$ is $t=0$, since the function is positive at $t=5/9$, independently of $\rho_2$.

Now, we study $\overline{F_{2,\rho_1,\rho_2}}$ in $[5/9,1]$:
\begin{align*}
\overline{F_{2,\rho_1,\rho_2}}'(t) \geq 0 \quad&\iff\quad
-\dfrac{25e^{M_2+2\rho_2+\frac{1}{2}}}{81\pi} + \dfrac{25e^{m_2-2\rho_2-\frac{1}{2}}}{81\pi} \left(\dfrac{9\cos(bt)}{5} +1 - \dfrac{ \sin(b)}{\pi}\right) \geq 0\\ &\iff\quad
e^{m_2-2\rho_2-\frac{1}{2}} \left(\dfrac{9\cos(bt)}{5} +1 - \dfrac{ \sin(b)}{\pi}\right) \geq e^{M_2+2\rho_2+\frac{1}{2}} \\ &\iff\quad
\cos(bt) \geq \frac{5}{9}\left( e^{M_2-m_2+4\rho_2+1} - 1 + \dfrac{ \sin(b)}{\pi} \right) =: \gamma_{\rho_2}.
\end{align*}
The latter inequality is solvable if and only if $\gamma_{\rho_2} \in [-1,\sup_{(5/9,1)}\cos(bt)]=[-1,\cos(b)]$, and this never happens independently of $\rho_2>0$. Indeed,
\begin{align*}
\gamma_{\rho_2}\leq \cos(b) \quad&\iff\quad e^{M_2-m_2+4\rho_2+1}-1+\frac{\sin(b)}{\pi} \leq \frac{9}{5}\cos(b) \\ &\iff\quad
e^{M_2-m_2+4\rho_2+1} \leq \frac{9}{5}\cos(b) + 1 - \frac{\sin(b)}{\pi} \approx2.643.
\end{align*}
Since $e^{M_2-m_2+4\rho_2+1}>e$, the latter inequality is always false, and this implies that $\overline{F_{2,\rho_1,\rho_2}}$ is decreasing in $[5/9,1]$, hence its absolute minimum in $[0,1]$ is $\overline{F_{2,\rho_1,\rho_2}}(0)=0$.

Now, we search for the absolute maximum points in $[0,1]$ of $\underline{F_{2,\rho_1,\rho_2}}$. We start with the interval $[0,5/9]$. By following the same idea we used for $\overline{F_{2,\rho_1,\rho_2}}$, we get that
\begin{align*}
\underline{F_{2,\rho_1,\rho_2}}'(t) \geq 0 \quad\iff\quad
\cos(bt) \geq \frac{e^{M_2-m_2+4\rho_2+1}}{9}\left(\dfrac{5 \sin(b)}{\pi} + 4\right) - \frac{4}{9} =: \gamma_{\rho_2}.
\end{align*}
Note that, for every $\rho_2>0$,
\[
\gamma_{\rho_2} > \frac{e}{9}\left(\dfrac{5 \sin(b)}{\pi} + 4\right) - \frac{4}{9} \approx 0.4811 > -1.
\]
At the same time,
\begin{align*}
\gamma_{\rho_2} < 1 \quad&\iff\quad e^{M_2-m_2+4\rho_2+1} < \frac{13}{\dfrac{5 \sin(b)}{\pi} + 4} \approx 4.2421 \\
&\iff\quad \rho_2 < \frac{1}{4}\log\left( \frac{13}{\dfrac{5 \sin(b)}{\pi} + 4} \right) - \frac{M_2-m_2+1}{4} =: \alpha \approx 0.09609.
\end{align*}
Therefore, we can solve $\underline{F_{2,\rho_1,\rho_2}}'(t) \geq 0$, for $t \in [0,5/9]$, if and only if $\rho_2<\alpha$ and, in such cases,
\begin{align*}
\max_{[0,5/9]} \underline{F_{2,\rho_1,\rho_2}} &= \max\left\{ \underline{F_{2,\rho_1,\rho_2}}(0), \underline{F_{2,\rho_1,\rho_2}}\left( \frac{5}{9\pi}\arccos\left(\frac{e^{M_2-m_2+4\rho_2+1}}{9}\left(\frac{5\sin(b)}{\pi}+4\right)-\frac{4}{9}\right) \right) \right\} \\
&= \max\left\{ 0, \underline{F_{2,\rho_1,\rho_2}}\left( \frac{5}{9\pi}\arccos\left(\frac{e^{M_2-m_2+4\rho_2+1}}{9}\left(\frac{5\sin(b)}{\pi}+4\right)-\frac{4}{9}\right) \right) \right\}.
\end{align*}
Numerically, one may see that $\underline{F_{2,\rho_1,\rho_2}}\left( \frac{5}{9\pi}\arccos\left(\frac{e^{M_2-m_2+4\rho_2+1}}{9}\left(\frac{5\sin(b)}{\pi}+4\right)-\frac{4}{9}\right) \right)>0$ for every $\rho_2<\alpha$, hence
\[
\max_{[0,5/9]} \underline{F_{2,\rho_1,\rho_2}} = \underline{F_{2,\rho_1,\rho_2}}\left( \frac{5}{9\pi}\arccos\left(\frac{e^{M_2-m_2+4\rho_2+1}}{9}\left(\frac{5\sin(b)}{\pi}+4\right)-\frac{4}{9}\right) \right)>0
\]
for every $\rho_2<\alpha$.

Now, let us focus on the interval $[5/9,1]$. Again, by following the same idea we have applied for $\overline{F_{2,\rho_1,\rho_2}}$, we have that
\begin{align*}
\overline{F_{2,\rho_1,\rho_2}}'(t) \geq 0 \quad&\iff\quad
\cos(bt) \geq \frac{5}{9}\left( e^{m_2-M_2-4\rho_2-1} + \dfrac{ \sin(b)}{\pi} - 1 \right) =: \gamma_{\rho_2}.
\end{align*}
Note that, for every $\rho_2>0$,
\[
\gamma_{\rho_2} < \frac{5}{9}\left( e^{-1} + \dfrac{ \sin(b)}{\pi} - 1 \right) \approx -0.4551 < 0.809 \approx \cos(b).
\]
At the same time, we have
\begin{align*}
\gamma_{\rho_2} > -1 \quad&\iff\quad e^{m_2-M_2-4\rho_2-1} + \dfrac{ \sin(b)}{\pi} - 1 > -\frac{9}{5} \\
&\iff\quad e^{m_2-M_2-4\rho_2-1} > -\frac{4}{5} - \dfrac{ \sin(b)}{\pi},
\end{align*}
and this is always true since, for every $\rho_2>0$,
\[
e^{m_2-M_2-4\rho_2-1} > 0 > -\frac{4}{5} - \dfrac{ \sin(b)}{\pi}.
\]
Therefore, we can solve $\overline{F_{2,\rho_1,\rho_2}}'(t) \geq 0$ for every $\rho_2>0$, and in particular the latter inequality is equivalent to say that
\[
t \geq 2\pi - \frac{5}{9\pi}\arccos{\left( \frac{5}{9}\left( e^{m_2-M_2-4\rho_2-1} + \dfrac{ \sin(b)}{\pi} - 1 \right) \right)}.
\]
Therefore we obtain $$\max_{[5/9,1]}\underline{F_{2,\rho_1,\rho_2}}=\max\left\{  \underline{F_{2,\rho_1,\rho_2}}(5/9),\underline{F_{2,\rho_1,\rho_2}}(1)\right\}=\max\left\{ \underline{F_{2,\rho_1,\rho_2}}(5/9),0\right\}.$$ Note that
\begin{align*}
\underline{F_{2,\rho_1,\rho_2}}(5/9)
&= \dfrac{5e^{m_2-2\rho_2-\frac{1}{2}}}{81\pi} \left(\dfrac{5 \sin(\pi)}{\pi} + \frac{20}{9}\right) - \dfrac{5e^{M_2+2\rho_2+\frac{1}{2}}}{81\pi}\left(\dfrac{5 \sin(b)}{\pi} + 4\right)\frac{5}{9}\\
&= \dfrac{100e^{m_2-2\rho_2-\frac{1}{2}}}{729\pi} - \dfrac{25e^{M_2+2\rho_2+\frac{1}{2}}}{729\pi}\left(\dfrac{5 \sin(b)}{\pi} + 4\right)\\
&\approx 0.0436 e^{m_2-2\rho_2-\frac{1}{2}} - 0.0334 e^{M_2+2\rho_2+\frac{1}{2}} < 0
\end{align*}
for every $\rho_2>0$. Therefore we obtain $\max_{[5/9,1]}\underline{F_{2,\rho_1,\rho_2}}=0$.

To summarize, we have
\[
\max_{[0,1]}\underline{F_{2,\rho_1,\rho_2}} = \max_{[0,5/9]}\underline{F_{2,\rho_1,\rho_2}} \geq 0,
\]
and this maximum is positive for every $\rho_2<\alpha\approx0.09609$. Thus, in this case, conditions $\ref{b}$ and $\ref{d}$ hold. In other words, we may apply Theorem~\ref{ThmG} if $(\rho_1,\rho_2) \in (0,+\infty)\times(0,\alpha)$ and in such cases there exist
\begin{gather*}
    (\lambda_{1,1},\lambda_{2,1}) \in (0,+\infty) \times (0,+\infty),\quad  (\lambda_{1,2},\lambda_{2,2}) \in (0,+\infty)\times(-\infty,0),\\
    (u_{1,1},u_{2,1})=(\phi_1+v_{1,1},\phi_2+v_{2,1}) \in \overline{K^{\phi_1}_{\rho_1}} \times \overline{B^{\phi_2}_{\rho_2}}, \text{ with } \| v_{1,1}\|_1=\rho_1, \| v_{2,1}\|_2=\rho_2,\\
    (u_{1,2},u_{2,2})=(\phi_1+v_{1,2},\phi_2+v_{2,2}) \in \overline{K^{\phi_1}_{\rho_1}} \times \overline{B^{\phi_2}_{\rho_2}}, \text{ with } \| v_{1,2}\|_1=\rho_1, \| v_{2,2}\|_2=\rho_2
\end{gather*}
such that, for $j=1,2$,
\[
\begin{cases}
    u_{1,j}(t) = \phi_1(t) + \lambda_{1,j} \int_0^1 k_1(t,s) f_1(s,u_{1,j}(s),u_{2,j}(s),H_1[u_{1,j},u_{2,j}](s)) \,ds,&t \in \left[ -\frac{\pi}{8}, 1+\frac{\pi}{8} \right], \\
    u_{2,j}(t) = \phi_2(t) + \lambda_{2,j} \int_0^1 k_2(t,s) f_2(s,u_{1,j}(s),u_{2,j}(s),H_2[u_{1,j},u_{2,j}](s)) \,ds, &t \in \left[ -\frac{\pi}{6}, 1+\frac{\pi}{16} \right]. \\
\end{cases}
\]
In Figure~\ref{fig:sol approx} we show a numerical approximation of these two couples of solutions.
\begin{figure}[h]
    \centering
    \includegraphics[width=0.7\linewidth]{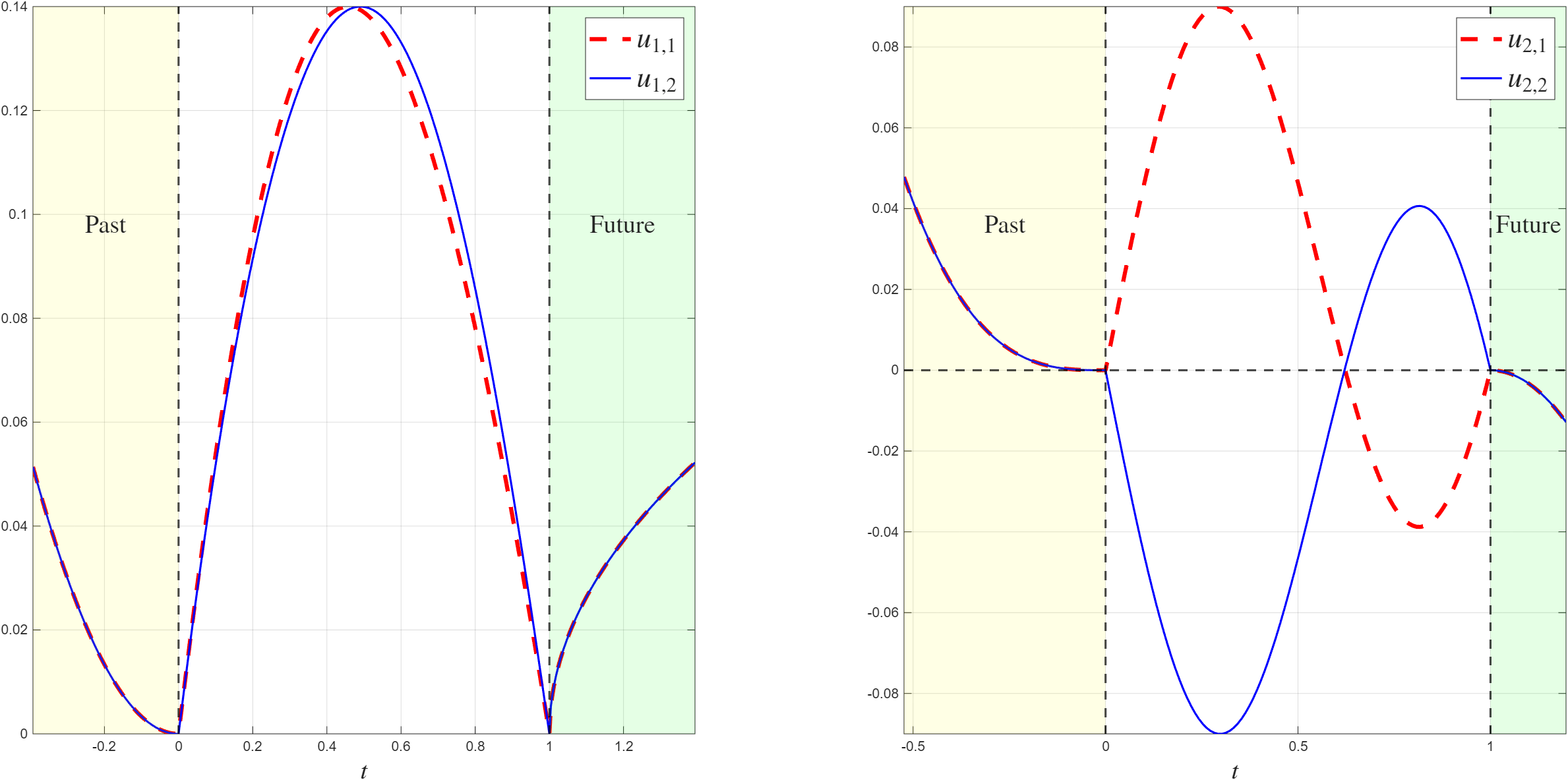}
    \caption{Approximation of the solutions $u_{i,j}$ for $\rho_1=0.14$ and $\rho_2=0.09$.}
    \label{fig:sol approx}
\end{figure}
Moreover, since we have conditions \ref{b} and \ref{d}, we also have, for $j=1,2$, the estimates
\begin{equation}\label{estimates1}
    \underline{\Gamma_{\rho_1}}:=\frac{\rho_1}{\max\{\|\underline{F_{1,\rho_1,\rho_2}}\|_1,\|\overline{F_{1,\rho_1,\rho_2}}\|_1\}} \leq |\lambda_{1,j}| \leq \frac{\rho_1}{\max_{\left[ -\frac{\pi}{8}, 1+\frac{\pi}{8} \right]}\underline{F_{1,\rho_1,\rho_2}}} =: \overline{\Gamma_{\rho_1}},
\end{equation}
and
\begin{equation}\label{estimates2}
    \underline{\Gamma_{\rho_2}}:=\frac{\rho_2}{\max\{\|\underline{F_{2,\rho_1,\rho_2}}\|_2,\|\overline{F_{2,\rho_1,\rho_2}}\|_2\}} \leq |\lambda_{2,j}| \leq \frac{\rho_2}{\max_{\left[ -\frac{\pi}{6}, 1+\frac{\pi}{16} \right]}\underline{F_{2,\rho_1,\rho_2}}} =: \overline{\Gamma_{\rho_2}},
\end{equation}
for every $\rho_1>0$ and $\rho_2 \in (0,\alpha)$. 

By using MATLAB, in Figures~\ref{fig:eigenapprox++} and \ref{fig:eigenapprox+-}, we illustrate the inequalities~\eqref{estimates1} and \eqref{estimates2} and provide a numerical approximation of $\lambda_{i,j}$ which is consistent with the theoretical estimates.
\begin{figure}[h]
    \centering
    \includegraphics[width=0.7\linewidth]{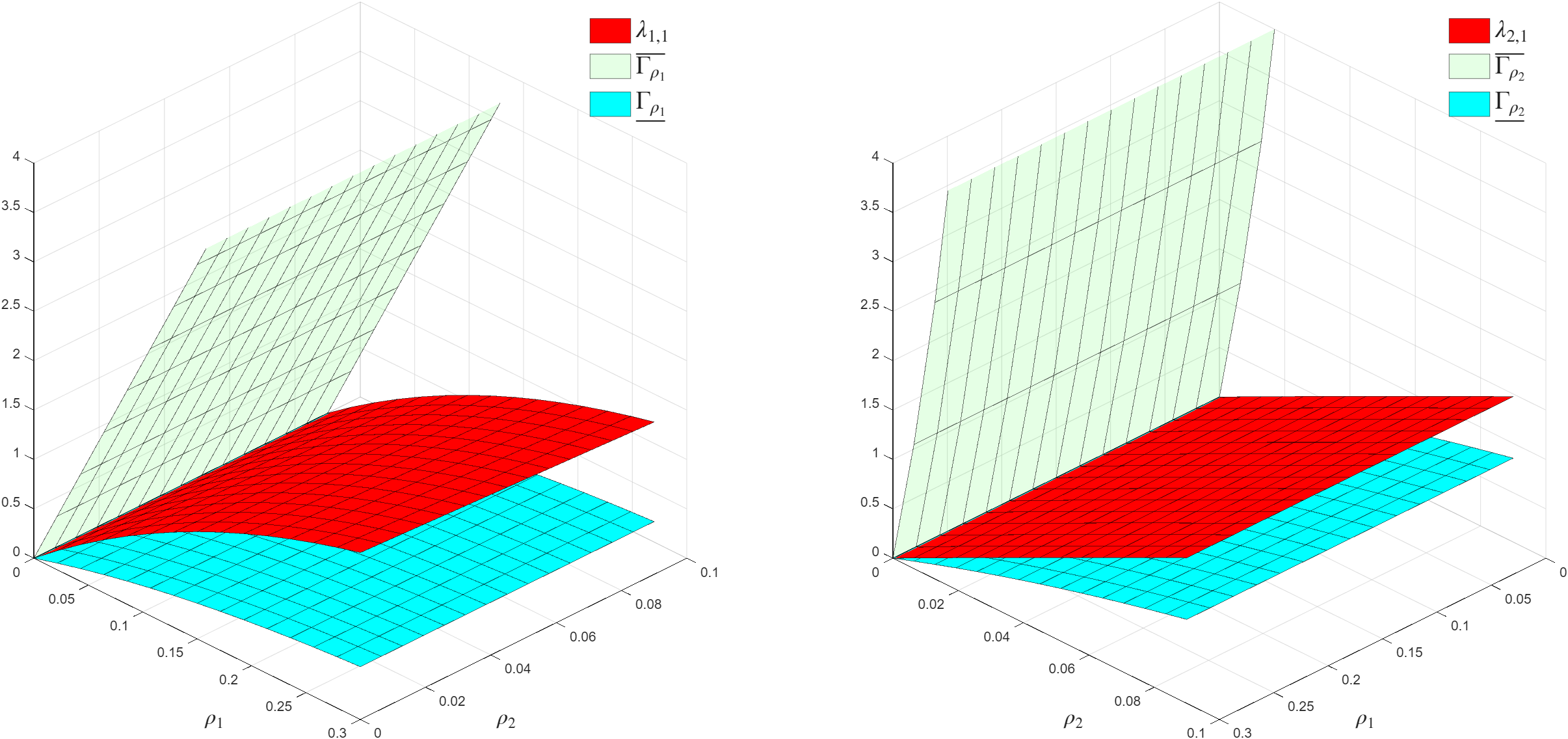}
    \caption{Location and approximation of $\lambda_{1,1}(\rho_1,\rho_2)$ and $\lambda_{2,1}(\rho_1,\rho_2)$.}
    \label{fig:eigenapprox++}
\end{figure}

\begin{figure}[h]
    \centering
    \includegraphics[width=0.7\linewidth]{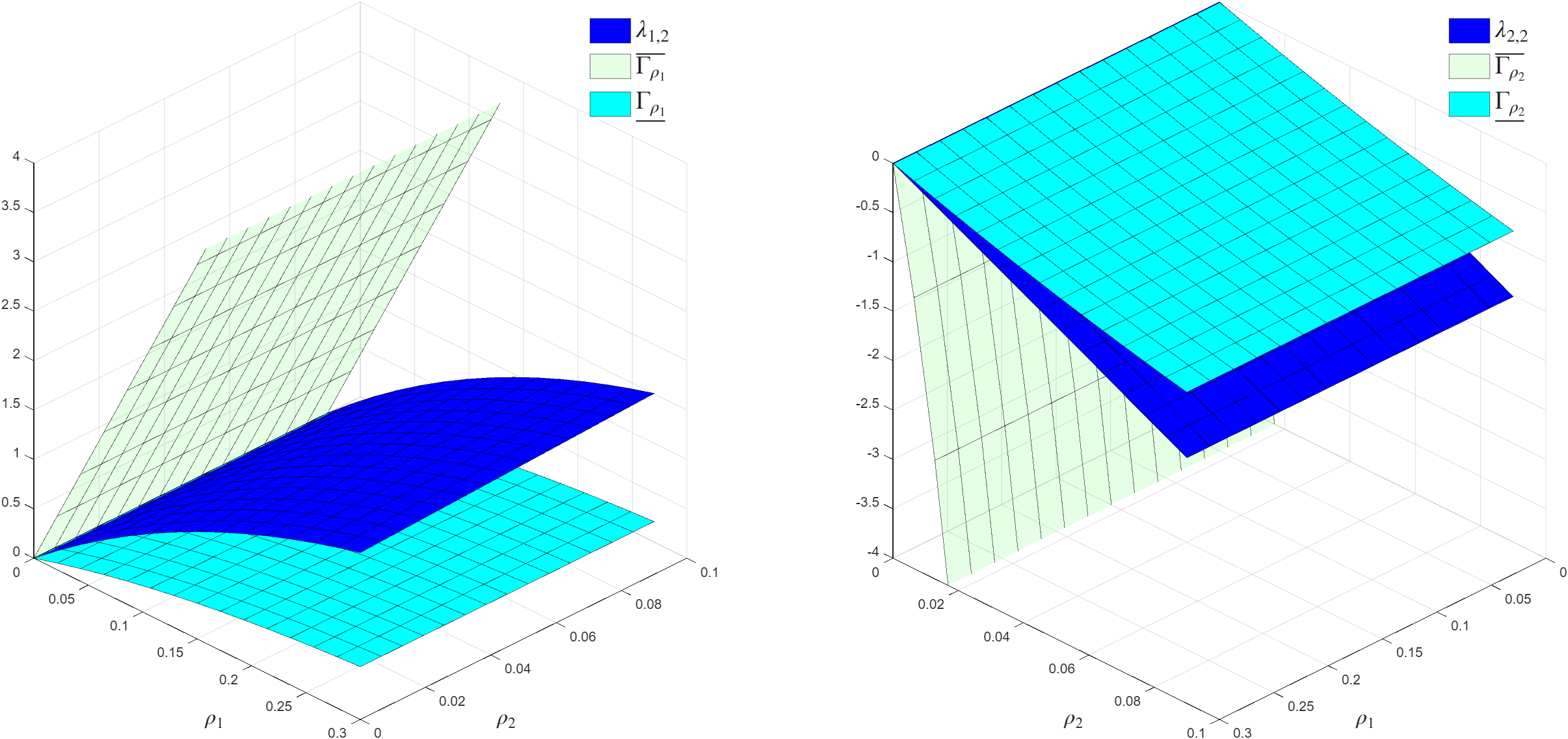}
    \caption{Location and approximation of $\lambda_{1,2}(\rho_1,\rho_2)$ and $\lambda_{2,2}(\rho_1,\rho_2)$.}
    \label{fig:eigenapprox+-}
\end{figure}

\end{ex}

\section*{Acknowledgements}
The authors would like to thank the anonymous referees for their careful reading of the manuscript and constructive comments.
The authors are members of the Gruppo Nazionale per l'Analisi Matematica, la Probabilit\`a e le loro Applicazioni (GNAMPA) of the Istituto Nazionale di Alta Matematica (INdAM).
G.~Infante and G. A. Veltri are members of the UMI Group TAA and of the The Research ITalian network on Approximation (RITA).

\section*{Funding}
A.~Calamai and G.~Infante have been partially supported by the GNAMPA. G.~Infante acknowledges the Italian Ministry of University and Research's grant [DD 170 24.09.2025 - DM MUR 737\_2021] supporting the University of Calabria for financing research in ``Aree Disciplinari Sociali e Umanistiche'', Department of Economics, Statistics and Finance [DD 248/2025 17.12.2025]. Project title: ``Actuarial and decision-making models to support the planning and development of healthcare welfare in Southern Italy and Calabria''.

\section*{Conflicts of interest}
The authors declare no conflict of interest.

\section*{Contribution statement}
All authors contributed equally to this manuscript.


\begin{thebibliography}{99}

\bibitem{ADV} J. Appell, E. De Pascale and A. Vignoli, \textit{Nonlinear spectral
theory}, Walter de Gruyter \& Co., Berlin,~2004.

\bibitem{Avramescu}	C. Avramescu, On a fixed point theorem, {\it St. Cerc. Mat.}, {\bf22(2)}, (1970) 215--221 (in Romanian). 

\bibitem{Benedetti}	
I. Benedetti, T. Cardinali and R. Precup, 
Fixed point-critical point hybrid theorems and application to systems with partial variational structure, {\it J. Fixed Point Theory Appl.}, {\bf 23} (2021), Paper No. 63, 19~pp.	

\bibitem{B-K-1922}
G. D. Birkhoff and O. D. Kellogg, Invariant points in function space, \textit{Trans. Amer. Math. Soc.}, \textbf{23} (1922), 96--115.

\bibitem{acgi2} A. Calamai and G. Infante, 
An affine Birkhoff--Kellogg type result in cones with applications to functional differential equations,
\textit{Math.\ Meth.\ Appl.\ Sci.}, \textbf{46} (2023), no.~11, 11897--11905.

\bibitem{acgijrl} A. Calamai, G. Infante, and J. Rodr\'iguez-L\'opez, 
Birkhoff--Kellogg type results in product spaces and their application to differential systems,
\textit{J. Fixed Point Theory Appl.}, \textbf{28} (2026), 37.

\bibitem{djeb2014}
S. Djebali and K. Mebarki, Fixed point index on translates of cones and applications, \textit{Nonlinear Stud.}, \textbf{21} (2014), 579--589. 

\bibitem{rub-rod-lms} R. Figueroa and R. L. Pouso, Minimal and maximal solutions to second-order boundary value problems with state-dependent deviating arguments, \textit{Bull. Lond. Math. Soc.}, \textbf{43} (2011), 164--174.

\bibitem{GrimmSchmitt}
L. J. Grimm and K. Schmitt, Boundary value problems for differential equations with deviating arguments, \textit{Aequationes Math.}, \textbf{4} (1970), 176--190.


\bibitem{guolak} D. Guo and V. Lakshmikantham,
\textit{Nonlinear problems in abstract cones}, Academic Press, Boston~1988. 




\bibitem{halelunel}
J. K. Hale and S. M. V. Lunel,
\textit{Introduction to Functional Differential Equations},
Springer Verlag, New York, 1993.

\bibitem{hartung}
F. Hartung, T. Krisztin, H.-O. Walther and J. Wu, Functional differential equations with state-dependent delays: theory and applications. \textit{Handbook of differential equations: ordinary differential equations Vol. III}, 435--545, Handb. Differ. Equ., Elsevier, Amsterdam, 2006.

\bibitem{hemedr}
H. R. Henr\'iquez, J. G. Mesquita and H. C. dos Reis, Existence results for abstract functional differential equations with infinite state-dependent delay and applications, \textit{Math. Ann.}, \textbf{388} (2024), 1817--1840.

\bibitem{gi-BK}
G. Infante, Eigenvalues of elliptic functional differential systems via a Birkhoff--Kellogg type theorem, \textit{Mathematics}, \textbf{9} (2021), n.~4, Paper No. 330, 9~pp.	

\bibitem{imap} G. Infante, M. Maciejewski and R. Precup, A topological approach to the existence
and multiplicity of positive solutions of $(p, q)$-Laplacian systems, \textit{{Dyn. Partial Differ. Equ.}}, \textbf{12}  (2015), 193--215.

\bibitem{InMaRo} G. Infante, G. Mascali and J. Rodr\'iguez-L\'opez, A hybrid Krasnosel'ski\u{\i}-Schauder fixed point theorem for systems,
\textit{Nonlinear Anal. Real World Appl.}, \textbf{80} (2024), 1--9.

\bibitem{Krasno}
M. A. Krasnosel'ski\u{i},
\textit{Positive solutions of operator equations}, Noordhoff, Groningen, 1964.



\bibitem{NtouyasTsamatos1}
S. K. Ntouyas and  P. Ch. Tsamatos, On well-posedness of boundary value problems involving deviating arguments, \textit{Funkc. Ekvacioj, Ser. Int.}, \textbf{35} (1992), 137--147.

\bibitem{NtouyasTsamatos2}
S. K. Ntouyas and  P. Ch. Tsamatos, 
Existence and uniqueness for second order boundary value problems, \textit{Funkc. Ekvacioj, Ser. Int.}, \textbf{38} (1995), 59--69.


\bibitem{Perov} A. I. Perov and A. V. Kibenko, On a certain general method for investigation of boundary value problems, \textit{Izv. Akad. Nauk SSSR}, \textbf{30} (1966), 249--264 (in Russian).
 
\bibitem{PrecupFPT} R. Precup, A vector version of Krasnosel'ski\u{\i}'s fixed point theorem in cones and positive periodic solutions of nonlinear systems, {\it J. Fixed Point Theory Appl.}, {\bf 2} (2007), 141--151.



\end{thebibliography}
\end{document}